\documentclass[12pt, reqno]{amsart}
\usepackage{amsfonts, amsmath, amsthm, amscd, amsfonts, amssymb, graphicx, color}
\usepackage{tikz-cd}
\usepackage[bookmarksnumbered, colorlinks, plainpages]{hyperref}
\newtheorem{theorem}{Theorem}[section]
\newtheorem{lemma}[theorem]{Lemma}
\newtheorem{proposition}[theorem]{Proposition}

\newtheorem{corollary}[theorem]{Corollary}
\newtheorem{definition}[theorem]{Definition}
\newtheorem{example}[theorem]{Example}

\usepackage[bookmarksnumbered, colorlinks, plainpages]{hyperref}
\hypersetup{colorlinks=true,linkcolor=blue, anchorcolor=green, citecolor=cyan, urlcolor=red, filecolor=magenta, pdftoolbar=true}

\usepackage{comment}
\excludecomment{mysection}

\begin{document}
	
\title[Absolute matrix order ideals in absolute matrix order unit spaces]{Absolute matrix order ideals in absolute Matrix order unit spaces}
\author{Amit Kumar}
	
\address{Discipline of Mathematics, School of Basic Sciences, Indian Institute of Technology Bhubaneswar, Argul, Bhubaneswar, Pin - 752050, Odisha (State), India.}

\email{\textcolor[rgb]{0.00,0.00,0.84}{amit231291280895@gmail.com}} 

	

\begin{abstract}
In this paper, we define and study absolute matrix order ideals in absolute matrix order unit spaces. As an application of absolute matrix order unit property, we construct some kinds of absolute matrix order ideals in absolute matrix order unit spaces. Later, we show that the Grothendieck group of a such kind of absolute matrix order ideal for order projections is a subgroup of Grothendieck group of corresponding absolute matrix order unit space for order projections.
\end{abstract}

\thanks{The author was financially supported by the Institute Post-doctoral Fellowship of IIT Bhubaneswar, India.}

\maketitle

\section{Introduction}

The concept of Hereditary sub-$C^*$-algebras is well known in a $C^*$-algebra theory. Let $\mathbb{A}$ be a $C^*$-algebra. A sub-$C^*$-algebra $\mathbb{B}$ of $\mathbb{A}$ is said to be Hereditary sub-$C^*$-algebra if it satisfies the following condition: $a\in \mathbb{B}$ for all $a\in \mathbb{A}^+$ and $b\in \mathbb{B}^+$ such that $a\leq b.$ This concept helped a lot in the study of certain notions in a $C^*$-algebra. For example, the notions of minimal, maximal, purely infinite and central projections in $C^*$-algebras are defined and studied in the terms of Hereditary sub-$C^*$-algebras. In this regard, we refer to see \cite{B98, B06, RLL00, WO93} and references therein.

A tiny version of the notion of Hereditary sub-$C^*$-algebras is also defined and studied in real ordered vector spaces that is known as the notion of order ideals. Let $\mathbb{X}$ be a real ordered vector space. A vector subpace $\mathbb{Y}$ of $\mathbb{X}$ is said to be an order ideal if it satisfies the following condition: $x\in \mathbb{Y}$ for all $x\in \mathbb{X}^+$ and $y\in \mathbb{Y}^+$ such that $x\leq y.$ Since the self adjoint part $\mathbb{A}_{sa}$ of $\mathbb{A}$ is a real ordered vector space, the notion of order ideals in ordered vector spaces is a genealization of the notion of Hereditary sub-$C^*$-algebra in $C^*$-algebras. For definition and more informations about order ideals, see \cite{MA71, J72, WN73}.

Order strucure of $C^*$-algebras has its own importance. It is very rich with certain properties and plays an important role in the classification of certain kind of $C^*$-algebras. Therefore the study of order structure becomes necessary. Many researchers started working in this direction. In this regard, we refer to see \cite{GN, Kad51, Kak, LA, LA73, A74, B54, B56, E64, JE64, KFN, GKP} and references therein. The notions of matricially normed spaces defined and studied by Ruan \cite{ER88, ZJR88}, and matrix ordered and matrix order unit spaces defined and studied by Choi and Effros \cite{CE77}, have high impact in this direction. The unital $C^*$-algebras are examples of such spaces. However, the non-unital versions of such spaces has been taken care by Karn \cite{KV97, KRV97, KV98, KV00} and Schreiner \cite{Schr98}. Such spaces are called $L^\infty$-matricially Riesz normed spaces and  matrix regular operator spaces by Karn and Schreiner respectively. The non-unital $C^*$-algebras are examples of such spaces. 

More than two decades, Karn is also working in direction. In fact, he is working on the order theoretic aspects of $C^*$-algebras. He has published several papers in this direction. We refer to see \cite{AK02, KV97, K10, K14, K16, K18}. In \cite{AK02}, he defined and studied order unit and matrix order unit properties in $L^\infty$-matricially Riesz normed spaces. He found that only projections satisfy order unit property in $C^*$-algebras \cite[Theorem 3.4]{AK02}. Thus, by help of order unit property, he characterized hereditary sub-$C^*$-algebras in $C^*$-algebras. He has also studied matrix order ideals in matrix order unit spaces which is found to be an abstract version of hereditary sub-$C^*$-algebras. He gave a characrterization of matrix order ideals in terms of order unit property \cite[Proposition 2.1]{AK02}. In \cite{K18}, he defined and studied a non-commutative vector lattice structure called as absolute order unit spaces. Unital $C^*$-algebras, unital $JB$-algebras and unital $AM$-spaces are some examples of such spaces. He defined and studied order unit property in absolute order unit spaces. He also defined and studied absolute version of order unit property in absolute order unit spaces namely absolute order unit property. Later, he introduced and studied the notion of order projections in absolute order unit spaces and gave a characterization in terms of absolute order unit property \cite[Proposition 6.1]{K18}. 

The author in the collaboration with Karn has also been starting working in this direction (see \cite{K19,PI19,KO21}). In \cite{K19}, he defined and studied the notion of absolute matrix order unit spaces. This notion is a matricial version of absolute order unit spaces. The unital $C^*$-algebras and $JB$-algebras are primary examples of these spaces. He defined and studied the notion of absolute value preserving maps on these spaces. He found that under certain conditions, these maps are isometries \cite[Theorems 3.3 and 4.6]{K19}. He also defined the notion of absolute order ideals and found that kernels of absolute value preserving maps forms absolute order ideals \cite[Theorem 2.7]{K19}. In \cite{PI19}, he generalized the notion of order projections in absolute matrix order unit spaces and, defined and characterized some variants of order projections \cite[Theorems 5.1 and 5.2]{PI19}. In \cite{KO21}, he defined and studied some equivalences on order projections in absolute matrix order unit spaces. By help of these equivalences, he formed $\mathcal{K}_0$-group, the Grothendieck group for order projections \cite[Theorem 4.8]{KO21}. He proved that $\mathcal{K}_0$-group is an ordered abelian and a covariant functor \cite[Corollaries 4.11 and 5.3]{KO21}.  

We start our research work in the continuation of our earlier work \cite{K19,PI19,KO21}. In this paper, we define and study absolute matrix order unit property in absolute matrix order unit spaces (Definition \ref{10} and Proposition \ref{11}). The absolute matrix order unit property in absolute matrix order unit spaces is a generalization of absolute order unit property in absolute order unit spaces. We also define the notion of absolute matrix order and absolute matrix order unit ideals in absolute order unit spaces (Definition \ref{12}). As an application of absolute matrix order unit propery, we construct absolute matrix order ideals in absolute matrix order unit spaces (Theorem \ref{4}). Further, as an application of absolute matrix order unit property, we also show that the quotients spaces by kernels of completely absolutely preserving maps can also be seen as absolute matrix order unit ideals (Theorem \ref{13}). Later, we prove that under ceratin condition the Grothendieck group of absolue matrix order unit ideal is subgroup of Grothendieck group of corresponding absolute matrix order unit space (Theorem \ref{14}). 

\section{Preliminaries}

Throughout this paper, $\mathbb{X}$ stands for a complex vector space and $\mathbb{M}_{l,m}(\mathbb{X})$ for the set of all the $l\times m$ matrices over $\mathbb{X}.$ Then $\mathbb{M}_{l,m}(\mathbb{X})$ becomes a vector space over $\mathbb{C}$ under the operations entry-wise addition and entry-wise scalar multiplication. If $\mathbb{X}=\mathbb{C},$ then we write $\mathbb{M}_{m,n}$ for $\mathbb{M}_{l,m}(\mathbb{C}).$ The zero element in $\mathbb{M}_{m,n}$ is denoted by $0^{l,m}$ and for $l=m,$ the zero element in $\mathbb{M}_l(\mathbb{X})$ is denoted by $0^l.$ Sometimes, we drop the subscripts and write zero element in $\mathbb{M}_{l,m}$ by $0$ only. In fact, $\mathbb{M}_l(\mathbb{X})$ is $\mathbb{M}_l$-bimodule under the following matrix multiplications: for $x =[x_{i,j}]\in \mathbb{M}_{m,n}(\mathbb{X}),$ we have $\varsigma x = \begin{bmatrix} \displaystyle \sum_{l=1}^m \varsigma_{i,l}x_{l,j}\end{bmatrix}\in \mathbb{M}_{l,n}(\mathbb{X})$ if $\varsigma=[\varsigma_{i,j}] \in \mathbb{M}_{l,m}$ and $x \varsigma = \begin{bmatrix} \displaystyle \sum_{l=1}^n x_{i,l} \varsigma_{l,j} \end{bmatrix}\in \mathbb{M}_{m,s}(\mathbb{X})$ if $\varsigma \in \mathbb{M}_{n,s}.$ Given $x \in \mathbb{M}_{l,m}(\mathbb{X}),y \in \mathbb{M}_{n,s}(\mathbb{X}),$ we write: 
\begin{center}
$x \circledast y = \begin{bmatrix} x & 0 \\ 0 & y \end{bmatrix}\in \mathbb{M}_{l+n,m+s}(\mathbb{X}).$
\end{center}
We also write: $x^l=x\circledast x\circledast \cdots \circledast x \in \mathbb{M}_{lm}(\mathbb{X})$ for every $x \in \mathbb{M}_m(\mathbb{X}).$

For origin of such matricial notions, we refer to see \cite{ZJR88}.

We start with the definition of absolutely matrix ordered spaces, a matricial version of absolutely ordered spaces. Absolutely ordered spaces were introduced and studied by Karn in \cite{K18}. It is very near to a lattice structure.  
\begin{definition}[\cite{K19}, Definition 4.1]\label{5}
Let $(\mathbb{X}, \lbrace \ \mathbb{M}_l(\mathbb{X})^+ \rbrace)$ be a matrix ordered space and also let $\vert\cdot\vert_{l,m}: \mathbb{M}_{l,m}(\mathbb{X}) \to \mathbb{M}_m(\mathbb{X})^+$ be functions for each pair $l, m \in \mathbb{N}$. For $l=m,$ we write $\vert\cdot\vert_{l,m} = \vert\cdot\vert_l.$ Then $\left(\mathbb{X}, \lbrace \mathbb{M}_m(\mathbb{X})^+ \rbrace, \lbrace \vert\cdot\vert_{l,m} \rbrace \right)$ is said to be an \emph{absolutely matrix ordered space}, if the following conditions are satisfied: 
\begin{enumerate}
\item[(1)] $(\mathbb{M}_m(\mathbb{X})_{sa},\mathbb{M}_m(\mathbb{X})^+, \vert\cdot\vert_m)$ is an absolutely ordered space for every $m \in \mathbb{N}.$
\item[(2)] For all $x \in \mathbb{M}_{m,n}(\mathbb{X}),y \in \mathbb{M}_{l,s}(\mathbb{X}), \varsigma_1 \in \mathbb{M}_{l,m}$ and $\varsigma_2 \in \mathbb{M}_{n,s},$ we have
\begin{enumerate}
\item[(a)] $\vert \varsigma_1 x \varsigma_2 \vert_{l,s} \leq \| \varsigma_1 \| \vert \vert x \vert_{m,n} \varsigma_2 \vert_{n,s}.$ 
\item[(b)] $\vert x \circledast y\vert_{m+l,n+s} = \vert x \vert_{m,n} \circledast \vert y \vert_{l,s}.$ 
\end{enumerate}
\end{enumerate} 
\end{definition}

\begin{example}[\cite{K19}, Example 4.4]
Given a $C^*$-algebra $\mathbb{A},$ we know that $\mathbb{M}_m(\mathbb{A})$ is also a $C^*$-algebra. Next, let $\mathbb{M}_m(\mathbb{A})^+$ be the set of all the positive elements contained in $\mathbb{M}_m(\mathbb{A}).$ Then $(\mathbb{A},\lbrace \mathbb{M}_m(\mathbb{A})^+\rbrace)$ becomes a matrix ordered space and $(\mathbb{A},\lbrace \mathbb{M}_m(\mathbb{A})^+\rbrace,\lbrace \vert \cdot \vert_{l,m}\rbrace)$ becomes an absolutely matrix ordered space, where each $\vert \cdot \vert_{l,m}:\mathbb{M}_{l,m}(\mathbb{A})\to \mathbb{M}_m(\mathbb{A})^+$ is defined by $\vert x\vert_{l,m}= \sqrt{x^*x}$ for $x\in \mathbb{M}_{l,m}(\mathbb{A}).$ 
\end{example}

The following result describes some properties of absolutely matrix ordered spaces.

\begin{proposition}[\cite{K19}, Proposition 4.2]\label{9}
In every absolutely matrix ordered space $(\mathbb{X},\lbrace \mathbb{M}_m(\mathbb{X})^+ \rbrace,\lbrace \vert\cdot\vert_{l,m} \rbrace),$ for each $l,m,n,s\in \mathbb{N}$ and $x \in \mathbb{M}_{m,n}(\mathbb{X}),\varsigma \in \mathbb{M}_{l,m}$ with $\varsigma^*\varsigma = \mathbb{I}_m~i.e.~\varsigma$ is an isometry, the following hold:
\begin{enumerate}
\item[(1)] $\vert \varsigma x \vert_{l,n} = \vert x \vert_{m,n}.$ 
\item[(2)] $\left\vert \begin{bmatrix} 0 & x \\ x^* & 0 \end{bmatrix} \right\vert_{m+n} = \vert x^* \vert_{n,m} \circledast \vert x \vert_{m,n}.$
\item[(3)] $\begin{bmatrix} \vert x^* \vert_{n,m} & x \\ x^* & \vert x \vert_{m,n} \end{bmatrix} \in \mathbb{M}_{m+n}(\mathbb{X})^+.$ 
\item[(4)] $\vert x \vert_{m,n} = \left\vert \begin{bmatrix} x \\ 0 \end{bmatrix} \right\vert_{m+l,n}$ and $\vert x \vert_{m,n} \circledast 0_s = \left\vert \begin{bmatrix} x & 0 \end{bmatrix} \right\vert_{m,n+s}.$
\item[(5)] $($\cite{PI19}, Lemma 3.1$).$ Moreover, if $l=m=n$ (in this case, $\varsigma$ is unitary), we also have: $\vert \varsigma^* x \varsigma \vert_n = \varsigma^* \vert x \vert_n \varsigma.$
\end{enumerate} 
\end{proposition}

Next, we write the definition of absolute matrix order unit spaces, a matricial version of absolute order unit spaces. Absolute order unit spaces were introduced and studied by Karn in \cite{K18}.

\begin{definition}[\cite{K19}, Definition 4.3]
Let $(\mathbb{X}, \lbrace \mathbb{M}_m(\mathbb{X})^+ \rbrace, e)$ be a matrix order unit space and also let $\vert \cdot \vert
_{l,m}:\mathbb{M}_{l,m}(\mathbb{X})\to \mathbb{M}_m(\mathbb{X})^+$ be a function for each pair $l,m \in \mathbb{N}.$ If $\left(\mathbb{X}, \lbrace \mathbb{M}_m(\mathbb{X})^+ \rbrace, \lbrace \vert \cdot \vert_{l,m} \rbrace \right)$ forms an absolutely matrix ordered space and $\perp = \perp_{\infty}^a$ on $\mathbb{M}_m(\mathbb{X})^+$ for all $m \in \mathbb{N},$ then we say that $(\mathbb{X}, \lbrace \mathbb{M}_m(\mathbb{X})^+ \rbrace, e)$ with $\lbrace \vert\cdot\vert_{l,m} \rbrace$ is an \emph{absolute matrix order unit space}. The absolute matrix order unit space is denoted by $(\mathbb{X}, \lbrace \mathbb{M}_m(\mathbb{X})^+ \rbrace, \lbrace \vert\cdot\vert_{l,m} \rbrace, e).$
\end{definition}

\begin{example}[\cite{K19}, Example 4.4]
Given any $C^*$-algebra $\mathbb{A},$ we have $\perp = \perp_\infty^a$ on $\mathbb{M}_m(\mathbb{A})^+$ for every $m \in \mathbb{N}.$ Therefore, every unital $C^*$-algebra forms an absolute matrix order unit space.
\end{example}

Now, we recall the definitions of order unit property and absolute order unit propertie introduced and studied by Karn in \cite{K18}.

\begin{definition}[\cite{K18}, Definition 5.1]
Let $(\mathbb{X},e)$ be an absolute matrix order unit space and $x \in \mathbb{M}_l(\mathbb{X})^+$. We write: $\mathbb{M}_l(\mathbb{X})_{sa}^x=\lbrace y\in \mathbb{M}_l(\mathbb{X})_{sa}: \epsilon x\pm y \in \mathbb{M}_l(\mathbb{X})^+ ~ \textrm{for some $\epsilon > 0$}~\rbrace.$ Then $x \in \mathbb{M}_l(\mathbb{X})^+$ is said to have the \emph{order unit property} in $\mathbb{M}_l(\mathbb{X})_{sa}$ provided $\pm y \leq \Vert y\Vert x$ for every $y\in \mathbb{M}_l(\mathbb{X})_{sa}^x.$  Moreover, $x$ is said to have the \emph{absolute order unit property} in $\mathbb{M}_l(\mathbb{X})_{sa}$ provided $\vert y\vert \leq \Vert y\Vert x$ for every $y\in \mathbb{M}_l(\mathbb{X})_{sa}^x.$ 
\end{definition}

The sets of all the elements having order unit property and absolute order unit property in  $\mathbb{M}_l(\mathbb{X})_{sa}$ are denoted by $\mathcal{M}_l(\mathbb{X})$ and $\vert \mathcal{M}\vert_l(\mathbb{X})$ respectively. For $l=1,$ we write: $\mathcal{M}_1(\mathbb{X})=\mathcal{M}(\mathbb{X})$ and $\vert \mathcal{M}\vert_1(\mathbb{X})=\vert \mathcal{M}\vert(\mathbb{X}).$

We explain the notions of order projections and partial isometries in absolute matrix order unit spaces. These notions are generalizations of the notions of projections and partial isometries in a $C^*$-algebra respectively .

\begin{definition}[\cite{PI19}, Definition 3.1]
Let $\mathbb{X}$ be an absolute matrix order unit space. We say that $x \in \mathbb{M}_n(\mathbb{X})$ is an \emph{order projection} provided $x \in \mathbb{M}_n(\mathbb{X})_{sa}$ with $\vert 2 x - e^n \vert_n = e^n.$ Next, if $x \in \mathbb{M}_{m,n}(\mathbb{X}),$ then we say that $x$ is a \emph{partial isometry} provided $\vert x \vert_{m,n} \in \mathbb{OP}_n(\mathbb{X})$ and $\vert x^* \vert_{n,m} \in \mathbb{OP}_m(\mathbb{X}).$ 
\end{definition}

The notations $\mathbb{OP}_n(\mathbb{X})$ and $\mathbb{PI}_{m,n}(\mathbb{X})$ stand for the sets of all the order projections and all the partial isometries in $\mathbb{M}_n(\mathbb{X})$ and $\mathbb{M}_{m,n}(\mathbb{X})$ respectively. We write: $\mathbb{PI}_{m,n}(\mathbb{X})=\mathbb{PI}_n(\mathbb{X})$ for $m=n.$ We also write: $\mathbb{OP}_1(\mathbb{X})=\mathbb{OP}(\mathbb{X})$ and $\mathbb{PI}_1(\mathbb{X})=\mathbb{PI}(\mathbb{X}).$

Next result recalls an elementary property of order projections. 

\begin{lemma}\cite[Proposition 3.2]{PI19}\label{6}
In an absolute matrix order unit space $\mathbb{X},$ we have: $p\circledast q \in \mathbb{OP}_{l+m}(\mathbb{X})$ if and only if  $p\in \mathbb{OP}_l(\mathbb{X})$ and $q \in \mathbb{OP}_m(\mathbb{X}).$ 
\end{lemma}

The following result provides a characterization of order projections in terms of order unit property and absolute order unit property.

\begin{proposition}\cite[Proposition 6.1]{K18}\label{7}
In an absolute matrix order unit space $(\mathbb{X},e),$ let $x \in \mathbb{M}_l(\mathbb{X})_{sa}$ such that $0 \le x \le e^l.$ Then the following statements are equivalent:
\begin{enumerate}
\item[(1)] $x\in \mathbb{OP}_l(\mathbb{X}).$
\item[(2)] $x$ and $e-x\in \mathcal{M}_l(\mathbb{X}).$
\item[(3)] $x$ and $e-x\in \vert\mathcal{M}\vert_l(\mathbb{X}).$
\end{enumerate}
\end{proposition}

\section{Order ideals in absolute matrix order unit spaces}

In this section, we recall matrix order unit property and define absolute matrix order unit property. These are matricial versions of order unit property and absolute order unit property which were introduced by Karn in \cite{K18}.

\begin{definition}\cite{AK02}\label{10}
Given $x \in \mathbb{X}^+$ in an absolute matrix order unit space $\mathbb{X},$ we say that $x$ has the matrix order unit property in $\mathbb{X}$ provided $x^l$ has the order unit property in $\mathbb{M}_l(\mathbb{X})_{sa}$ for any $l\in \mathbb{N}.$ Further, if $x^l$ has the absolute order unit property in $\mathbb{M}_l(\mathbb{X})_{sa}$ for any $l\in \mathbb{N},$ we say that $x$ has the absolute matrix order unit property in $\mathbb{X}.$ 
\end{definition}

The sets of all the elements having matrix order unit property and absolute matrix order unit property are denoted by $\mathcal{M}_\infty(\mathbb{X})$ and $\vert \mathcal{M}\vert_\infty(\mathbb{X})$ respectively.

The following result elaborates some properties enjoyed by matrix order and absolute matrix order unit properties.

\begin{proposition}\label{11}
Given $x\in \mathbb{X}^+$ in an absolute matrix order unit space $(\mathbb{X},e),$ the following statements hold:  
\begin{enumerate}
\item[(1)] $x\in \mathcal{M}_\infty(\mathbb{X})$ if and only if, we have $\begin{bmatrix} \Vert y\Vert_l x^l & y\\ y^* & \Vert y\Vert_l x^l\end{bmatrix} \in \mathbb{M}_{2l}(\mathbb{X})^+$ whenever $\begin{bmatrix} \epsilon x^l & y\\ y^* & \epsilon x^l\end{bmatrix} \in \mathbb{M}_{2l}(\mathbb{X})^+$ for some $l \in \mathbb{N}, \epsilon > 0$ and $y\in \mathbb{M}_l(\mathbb{X}).$  
\item[(2)] $x\in \vert \mathcal{M}\vert_\infty(\mathbb{X})$ if and only if, we have $\Vert y\Vert_l x^l- \vert y\vert_l~\textrm{and}~\Vert y\Vert_l x^l - \vert y^*\vert_l \in \mathbb{M}_l(\mathbb{X})^+$ whenever $\begin{bmatrix} \epsilon x^l & y\\ y^* & \epsilon x^l\end{bmatrix} \in \mathbb{M}_{2l}(\mathbb{X})^+$ for some $l \in \mathbb{N}, \epsilon > 0$ and $y\in \mathbb{M}_l(\mathbb{X}).$  
\end{enumerate}
Moreover, for $\Vert x\Vert \leq 1,$ we also have:
\begin{enumerate}
\item[(3)]$x$ and $e-x\in \mathcal{M}(\mathbb{X})$ if and only if $x$ and $e-x\in \mathcal{M}_\infty(\mathbb{X}).$ 
\item[(4)] $x$ and $e-x\in \vert \mathcal{M}\vert(\mathbb{X})$ if and only if $x$ and $e-x\in \vert \mathcal{M}\vert_\infty(\mathbb{X}).$ 
\end{enumerate}
\end{proposition}

\begin{proof}
\begin{enumerate}
\item[(1)] Let $x\in \mathcal{M}_\infty(\mathbb{X}).$ Assume that for $y\in \mathbb{M}_l(\mathbb{X}),$ there exists $\epsilon > 0$ satisfying $\begin{bmatrix} \epsilon x^l & y\\ y^* & \epsilon x^n\end{bmatrix} \in \mathbb{M}_{2l}(\mathbb{X})^+ ~ \textrm{for some} ~ \epsilon \in \mathbb{R}$ with $\epsilon >0.$ In this case, we also have $\begin{bmatrix} \epsilon x^n & -y\\ -y^* & \epsilon x^n\end{bmatrix}=\begin{bmatrix} -\mathbb{I}_l & 0\\ 0 & \mathbb{I}_l\end{bmatrix}\begin{bmatrix} \epsilon x^l & y\\ y^* & \epsilon x^n\end{bmatrix} \begin{bmatrix} -\mathbb{I}_l & 0\\ 0 & \mathbb{I}_l\end{bmatrix}\in \mathbb{M}_{2l}(\mathbb{X})^+.$ Since $x^{2l}\in \mathcal{M}_{2l}(\mathbb{X})$, we get that $\pm \begin{bmatrix}0 & y \\ y^* & 0\end{bmatrix} \leq \left \Vert \begin{bmatrix}0 & y \\ y^* & 0\end{bmatrix} \right \Vert_{2l} \begin{bmatrix}x^l & 0\\ 0 & x^l \end{bmatrix} = \Vert y \Vert_l \begin{bmatrix}x^l & 0\\ 0 & x^l \end{bmatrix}.$ Thus $\begin{bmatrix} \Vert y\Vert_l x^l & y\\ y^* & \Vert y\Vert_l x^l\end{bmatrix} \in \mathbb{M}_{2l}(\mathbb{X})^+.$

Next, we prove the converse part. For that let $y\in \mathbb{M}_l(\mathbb{X})_{sa}$ satisfying $\pm y \leq \epsilon x^l~ \textrm{for some}~ \epsilon >0.$ Then $\begin{bmatrix} \epsilon x^l & y\\ y & \epsilon x^l\end{bmatrix} = \frac{1}{2} \left ( \begin{bmatrix} \mathbb{I}_l \\ \mathbb{I}_l \end{bmatrix} (\epsilon x^l + y) \begin{bmatrix} \mathbb{I}_l & \mathbb{I}_l \end{bmatrix} \right ) + \frac{1}{2} \left (\begin{bmatrix} \mathbb{I}_l \\ - \mathbb{I}_l \end{bmatrix} (\epsilon x^n - y) \begin{bmatrix} \mathbb{I}_l & - \mathbb{I}_l \end{bmatrix} \right )\in \mathbb{M}_{2l}(\mathbb{X})^+.$ By assumption, we conclude that $\begin{bmatrix} \Vert y\Vert_l x^l & y\\ y & \Vert y\Vert_l x^l\end{bmatrix} \in \mathbb{M}_{2l}(\mathbb{X})^+.$ In this case, we also conclude that $$\Vert y\Vert x^l \pm y = \frac{1}{2} \left ( \begin{bmatrix} \mathbb{I}_l & \pm \mathbb{I}_l \end{bmatrix} \begin{bmatrix} \Vert y\Vert_l x^l & y\\ y & \Vert y\Vert_l x^l\end{bmatrix} \begin{bmatrix} \mathbb{I}_l \\ \pm \mathbb{I}_l \end{bmatrix} \right ) \in \mathbb{X}^+.$$ Thus $x^l\in \mathcal{M}_l(\mathbb{X})$ for every $l\in \mathbb{N}.$ Hence $x\in \mathcal{M}_\infty(\mathbb{X}).$ 

\item[(2)] Let $x\in \vert\mathcal{M}\vert_\infty(\mathbb{X}).$ Also, let $y\in \mathbb{M}_l(\mathbb{X})$ such that $\begin{bmatrix} \epsilon x^l & y\\ y^* & \epsilon x^n\end{bmatrix} \in \mathbb{M}_{2l}(\mathbb{X})^+$ for some $\epsilon >0.$ Then, we also have $\begin{bmatrix} \epsilon x^n & -y\\ -y^* & \epsilon x^n\end{bmatrix}\in \mathbb{M}_{2l}(\mathbb{X})^+.$ Since $x^{2l}\in \vert \mathcal{M}\vert_{2l}(\mathbb{X})$, we get that $\begin{bmatrix} \vert y^*\vert_l & 0\\ 0 & \vert y\vert_l\end{bmatrix}=\left \vert \begin{bmatrix}0 & y \\ y^* & 0\end{bmatrix}\right \vert_{2l} \leq \left \Vert \begin{bmatrix}0 & y \\ y^* & 0\end{bmatrix} \right \Vert_{2l} \begin{bmatrix}x^l & 0\\ 0 & x^l \end{bmatrix} = \Vert y \Vert_l \begin{bmatrix}x^l & 0\\ 0 & x^l \end{bmatrix}.$ Thus $\begin{bmatrix} \Vert y\Vert_l x^l-\vert y^*\vert_l & 0\\ 0 & \Vert y\Vert_l x^l-\vert y\vert_l \end{bmatrix} \in \mathbb{M}_{2l}(\mathbb{X})^+$ so that \begin{center}$\Vert y\Vert_l x^l-\vert y^*\vert_l = \begin{bmatrix}\mathbb{I}_l & 0 \end{bmatrix} \begin{bmatrix} \Vert y\Vert_l x^l-\vert y^*\vert_l & 0\\ 0 & \Vert y\Vert_l x^l-\vert y\vert_l \end{bmatrix} \begin{bmatrix}\mathbb{I}_l \\ 0 \end{bmatrix}\in \mathbb{M}_l(\mathbb{X})^+$\end{center} and \begin{center} $\Vert y\Vert_l x^l-\vert y\vert_l = \begin{bmatrix} 0 & \mathbb{I}_l \end{bmatrix} \begin{bmatrix} \Vert y\Vert_l x^l-\vert y^*\vert_l & 0\\ 0 & \Vert y\Vert_l x^l-\vert y\vert_l \end{bmatrix} \begin{bmatrix} 0 \\ \mathbb{I}_l \end{bmatrix}\in \mathbb{M}_l(\mathbb{X})^+.$ \end{center}

Now, let's prove the converse part. Let $y\in \mathbb{M}_l(\mathbb{X})_{sa}$ such that $\pm y \leq \epsilon x^l~ \textrm{for some}~ \epsilon >0.$ Then $\begin{bmatrix} \epsilon x^l & y\\ y & \epsilon x^l\end{bmatrix} \in \mathbb{M}_{2l}(\mathbb{X})^+.$ By assumption, we get that $\Vert y\Vert_l x^l - \vert y\vert_l \in \mathbb{M}_l(\mathbb{X})^+.$ Thus $x^l\in \vert \mathcal{M}\vert_l(\mathbb{X})$ for every $l\in \mathbb{N}.$ Hence $x\in \vert \mathcal{M}\vert_\infty(\mathbb{X}).$

\item[(3)] Suppose that $\Vert x\Vert \leq e.$ In this case, we have $0\leq x\leq e$ and consequently $0\leq e-x\leq e.$ By Lemma \ref{6} and Proposition \ref{7}, $x$ and $e-x \in \mathcal{M}(\mathbb{X})$ if and only if $x \in \mathbb{OP}(\mathbb{X})$ if and only if $x^l \in \mathbb{OP}(\mathbb{X})$ if and only if $x^l$ and $e^l-x^l\in \mathcal{M}_l(\mathbb{X})$ for every $l\in \mathbb{N}.$ Thus $x$ and $e-x \in \mathcal{M}(\mathbb{X})$ if and only if $x$ and $e-x \in \mathcal{M}_\infty(\mathbb{X}).$ 
\end{enumerate}
Now, using Lemma \ref{6} and Proposition \ref{7}, and repeating the proof of (3), we can also prove (4). 
\end{proof}

\begin{definition}\label{12}
Let $\left(\mathbb{X}, \lbrace \mathbb{M}_m(\mathbb{X})^+ \rbrace, \lbrace \vert\cdot\vert_{l,m} \rbrace \right)$ be an absolutely matrix ordered space and $\mathbb{Y}$ be a subspace of $\mathbb{X}.$ For each $l\in \mathbb{N},$ we write: $\mathbb{M}_l(\mathbb{Y})_{sa}=\mathbb{M}_l(\mathbb{X})_{sa}\cap \mathbb{M}_l(\mathbb{Y})~\textrm{and}~\mathbb{M}_l(\mathbb{Y})^+=\mathbb{M}_l(\mathbb{X})^+\cap \mathbb{M}_l(\mathbb{Y}).$ Then $\mathbb{Y}$ is said to be absolute matrix order ideal of $\mathbb{X},$ if the following two conditions are satisfied:
\begin{enumerate}
\item[(1)] $(\mathbb{M}_l(\mathbb{Y})_{sa},\mathbb{M}_l(\mathbb{Y})^+)$ is an order ideal of $(\mathbb{M}_l(\mathbb{X})_{sa},\mathbb{M}_l(\mathbb{X})^+)$ for each $l\in \mathbb{N}.$
\item[(2)] $\vert \cdot \vert_{l,m}:\mathbb{M}_{l,m}(\mathbb{Y})\to \mathbb{M}(\mathbb{Y})_m^+$ is well defined for each pair of $l,m\in \mathbb{N}.$ In other words, $\left(\mathbb{Y}, \lbrace \mathbb{M}_m(\mathbb{Y})^+ \rbrace, \lbrace \vert\cdot\vert_{l,m} \rbrace \right)$ is itself an absolutely matrix ordered space.
\end{enumerate}
Further, let $(\mathbb{X},e)$ be an absolute matrix order unit space. Then $\mathbb{Y}$ is said to be absolute matrix order unit  ideal, if along with condition (1), the following additional condition is also satisfied: 
\begin{enumerate}
\item[(3)] $(\mathbb{Y},\lbrace \mathbb{M}_m(\mathbb{Y})^+\rbrace,\lbrace \vert\cdot\vert_{l,m} \rbrace,y)$ is an absolute matrix order unit space such that $\Vert \cdot \Vert_m^y=\Vert \cdot \Vert_m$ on $\mathbb{M}_m(\mathbb{Y}),$ where $\Vert \cdot \Vert_m~\textrm{and}~\Vert \cdot \Vert_m^y$ are the matrix norms determined by the order units $e$ and $y$ respectively for each $m\in \mathbb{N}.$
\end{enumerate}
\end{definition}

\begin{example}\label{8}
Let $\mathbb{A}$ be a unital $C^*$-algebra with unity element $1$ and $p$ be a non-zero projection in $\mathbb{A}.$ Put $\mathbb{A}_p=\lbrace pap:a \in \mathbb{A}\rbrace.$ Then $(\mathbb{A}_p,p)$ is an absolute matrix order unit ideal of $\mathbb{A}.$ To see this, we verify only order ideal condition as other conditions are routine to verify $($we refer to see \cite{RJ83}$).$ Now, let $a,b\in \mathbb{A}$ be such that $0\leq b \leq pap.$ Then $(1-p)b(1-p)=0$ so that $\Vert\sqrt{b}(1-p) \Vert^2=0.$ In this case, $\sqrt{b}(1-p)=0$ and we get that $b=bp$ and $b=b^*=(bp)^*=pb.$ Thus $b=pbp\in \mathbb{A}_p.$
\end{example}

Next result generalize the last example in absolute matrix order unit spaces for more general elements.

\begin{theorem}\label{4}
In an absolute matrix order unit space $\mathbb{X},$ let $x\in \vert \mathcal{M}\vert_\infty(\mathbb{X})$ such that $\Vert x\Vert=1.$ We write: $$\mathbb{X}_x = \left \lbrace y \in \mathbb{X} :\begin{bmatrix} \epsilon x & y \\ y^* & \epsilon x \end{bmatrix} \in \mathbb{M}_2(\mathbb{X})^+ ~ \textrm{for some}~ \epsilon > 0\right \rbrace$$ and $$\mathbb{M}_{l,m}(\mathbb{X})_{x^l,x^m} = \left\lbrace y \in \mathbb{M}_{l,m}(\mathbb{X}): \begin{bmatrix} \epsilon x^l & y \\ y^* & \epsilon x^m\end{bmatrix} \in \mathbb{M}_{l+m}(\mathbb{X})^+~\textrm{for some}~\epsilon>0\right\rbrace.$$ For $l=m,$ write: $\mathbb{M}_{l,m}(\mathbb{X})_{x^l,x^m} = \mathbb{M}_l(\mathbb{X})_{x^l}.$ The following statements are true: 
\begin{enumerate}
\item[(1)] $\mathbb{M}_{l,m}(\mathbb{X}_x) = \mathbb{M}_{l,m}(\mathbb{X})_{x^l,x^m}$. 
\item[(2)] $\mathbb{M}_l(\mathbb{X}_x)^+ = \mathbb{M}_l(\mathbb{X}_x) \bigcap \mathbb{M}_l(\mathbb{X})^+$ forms a proper cone. 
\item[(3)] $x$ forms an order unit for $\mathbb{X}_x.$
\item[(4)] For every $y\in \mathbb{M}_l(\mathbb{X}_x),$ we write: $$\Vert y\Vert_l^x = \inf\left\lbrace \epsilon > 0: \begin{bmatrix}\epsilon x^l & y \\ y^* & \epsilon x^l\end{bmatrix} \in \mathbb{M}_{2l}(\mathbb{X})^+\right\rbrace.$$ It turns out that $\Vert \cdot\Vert_l^x=\Vert \cdot\Vert_l$ on $\mathbb{M}_l(\mathbb{X}_x)$ for every $l\in \mathbb{N}.$
\item[(5)] $\vert y \vert_{l,m} \in \mathbb{M}_m(\mathbb{X}_x)^+$ for every $y \in \mathbb{M}_{l,m}(\mathbb{X}_x).$ 
\end{enumerate}
 
In this case, $(\mathbb{X}_x,\lbrace \mathbb{M}_m(\mathbb{X}_x)^+\rbrace,\lbrace\vert \cdot \vert_{l,m}\rbrace, x)$ forms
an absolute matrix order unit ideal of $\mathbb{X}.$  
\end{theorem}

\begin{proof}
It is routine to verify (2). Now, let's prove the other statements.
\begin{enumerate}
\item[(1)] Let $y=[y_{i,j}] \in \mathbb{M}_{l,m}(\mathbb{X}_x).$ There exists $\epsilon_{i,j}> 0$ satisfying $\begin{bmatrix} \epsilon_{i,j}x & y_{i,j} \\ y_{i,j}^* & \epsilon_{i,j}x \end{bmatrix} \in \mathbb{M}_2(\mathbb{X})^+$ for every $y_{i,j} \in \mathbb{X}_p.$  Put $\epsilon_0 = \max \lbrace \epsilon_{i,j} : 1 \leq i \leq l, 1 \le j \leq m \rbrace.$ Then $\epsilon_0 > 0$ and $\begin{bmatrix} \epsilon_0 x & y_{i,j} \\ y_{i,j}^* & \epsilon_0 x \end{bmatrix} \in \mathbb{M}_2(\mathbb{X})^+$ for every pair of $i$ and $j.$ 

Next, let $\varsigma_{i,j} \in \mathbb{M}_{l+m,2}$ such that 

\[ 
   \varsigma_{i,j} = \begin{cases}
    1 & \text{at (i,1) and (l+j,2)}  \\
    0 & \text{elsewhere}.
   \end{cases}
\]

so that 

$$\begin{bmatrix} m\epsilon_0 x^l & y \\ y^* & l \epsilon_0 x^m\end{bmatrix} = \displaystyle \sum_{i,j}\varsigma_{i,j}  \begin{bmatrix} \epsilon_0 x & y_{i,j} \\ y_{i,j}^* & \epsilon_0 x \end{bmatrix} \varsigma_{ij}^* \in \mathbb{M}_{l+m}(\mathbb{X})^+.$$ 

For $\epsilon = \max \lbrace m, l \rbrace \epsilon_0,$ we get that $$\begin{bmatrix}\epsilon x^l & y \\ y^* & \epsilon x^m\end{bmatrix} \in \mathbb{M}_{l+m}(\mathbb{X})^+.$$ Thus $y \in \mathbb{M}_{l,m}(\mathbb{X})_{x^l,x^m}.$ 

Now, we prove the converse part. For that let $y \in \mathbb{M}_{l,m}(\mathbb{X})_{x^l,x^m}.$ There exists $\epsilon >0$ satisfying $\begin{bmatrix}\epsilon x^l & y \\ y^* & \epsilon x^m\end{bmatrix}\in \mathbb{M}_{l+m}(\mathbb{X})^+.$ Then $$\begin{bmatrix}\epsilon x & y_{i,j} \\ y_{i,j}^* & \epsilon x\end{bmatrix} = \varsigma_{i,j}^*  \begin{bmatrix} \epsilon x^l & y \\ y^* & \epsilon x^m \end{bmatrix} \varsigma_{i,j} \in \mathbb{M}_2(\mathbb{X})^+.$$ Thus $y_{i,j} \in \mathbb{X}_x$ for every $i,j$ and consequently $\mathbb{M}_{l,m}(\mathbb{X}_x) = \mathbb{M}_{l,m}(\mathbb{X})_{x^l,x^m}.$

\item[(3)] By (1), for every $y \in \mathbb{M}_l(\mathbb{X}_x),$ there exists $\epsilon >0$ satisfying $\begin{bmatrix}\epsilon x^l & y \\ y^* & \epsilon x^l\end{bmatrix} \in \mathbb{M}_{2l}(\mathbb{X})^+.$ Thus $x$ forms an order unit for $\mathbb{X}_x.$
\item[(4)] Let $y\in \mathbb{M}_l(\mathbb{X}).$ Since $x\in \vert \mathcal{M}\vert_\infty(\mathbb{X}),$ we have 

$$\pm \begin{bmatrix}0&y\\y^* & 0\end{bmatrix} \leq  \left \Vert\begin{bmatrix}0&y\\y^* & 0\end{bmatrix}\right \Vert_{2l}\begin{bmatrix}x^l & 0\\0 & x^l\end{bmatrix}=\Vert y\Vert_l \begin{bmatrix}x^l & 0\\0 & x^l\end{bmatrix}.$$ Thus $\begin{bmatrix}\Vert y\Vert_l x^l & y\\y^* & \Vert y\Vert_l x^l\end{bmatrix}\in \mathbb{M}_{2l}(\mathbb{X})^+$ so that $\Vert y\Vert_l^x\leq \Vert y\Vert_l.$

Now, $x\in \mathbb{X}^+$ with $\Vert x\Vert=1.$ Then $\begin{bmatrix}x^l & 0\\0 & x^l\end{bmatrix}\in \mathbb{M}_{2l}(\mathbb{X})^+$ and $\left \Vert \begin{bmatrix}x^l & 0\\0 & x^l\end{bmatrix}\right\Vert = 1$ so that $\begin{bmatrix}x^l & 0\\0 & x^l\end{bmatrix}\leq \begin{bmatrix}e^l & 0\\0 & e^l\end{bmatrix}.$ As $x$ is order unit for $\mathbb{X}_x,$ we have $$\pm \begin{bmatrix}0&y\\y^* & 0\end{bmatrix} \leq  \left \Vert\begin{bmatrix}0&y\\y^* & 0\end{bmatrix}\right \Vert_{2l}^x\begin{bmatrix}x^l & 0\\0 & x^l\end{bmatrix}=\Vert y\Vert_l^x \begin{bmatrix}x^l & 0\\0 & x^l\end{bmatrix}$$ and consequently $$\pm \begin{bmatrix}0&y\\y^* & 0\end{bmatrix} \leq \Vert y\Vert_l^x \begin{bmatrix}e^l & 0\\0 & e^l\end{bmatrix}.$$ Thus $\begin{bmatrix}\Vert y\Vert_l^x e^l & y\\y^* & \Vert y\Vert_l^x e^l\end{bmatrix}\in \mathbb{M}_{2l}(\mathbb{X})^+$ so that $\Vert y\Vert_l\leq \Vert y\Vert_l^x.$ Hence $\Vert y\Vert_l^x=\Vert y\Vert_l$ for every $y\in \mathbb{M}_l(\mathbb{X}_x).$
 
\item[(5)] Let $y \in \mathbb{M}_{l,m}(\mathbb{X}_x).$ For $x\in \vert \mathcal{M}\vert_\infty(\mathbb{X})$ and $\epsilon = \left\Vert \begin{bmatrix} 0 & y \\ y^* & 0\end{bmatrix} \right\Vert_{m+n},$ we have $$\begin{bmatrix} \vert y^*\vert_{m,l} & 0 \\ 0 & \vert y\vert_{l,m}\end{bmatrix}=\left\vert \begin{bmatrix} 0 & y \\ y^* & 0\end{bmatrix}\right\vert_{l+m} \leq \epsilon \begin{bmatrix} x^l & 0 \\ 0 & x^m \end{bmatrix}.$$ Then $\vert y\vert_{l,m} \leq \epsilon x^m$ so that $\vert y\vert_{l,m} \in \mathbb{M}_m(\mathbb{X}_x)^+.$ Thus, the maps $\vert \cdot \vert_{l,m}:\mathbb{M}_{l,m}(\mathbb{X}_x)\to \mathbb{M}_m(\mathbb{X}_x)^+$ are well defined.  
\end{enumerate}

It is routine to verify $\perp = \perp_\infty^a$ on $\mathbb{M}_l(\mathbb{X}_x)^+$ for every $l \in \mathbb{N}.$ 

Hence $(\mathbb{X}_x,\lbrace \mathbb{M}_l(\mathbb{X}_x)^+\rbrace,x,\lbrace\vert \cdot \vert_{l,m}\rbrace)$ forms an absolute matrix order unit ideal of $\mathbb{X}.$
\end{proof} 

The following result shows that matrix norm in absolute matrix order unit spaces takes two discrete values $0$ and $1$ whenever it is restricted to the order projections. 

\begin{proposition}\label{1}
In an absolute matrix order unit space $\mathbb{X},$ let $p \in \mathbb{OP}_l(\mathbb{X})$ such that $p\neq 0.$ Then $\Vert p \Vert_l = 1.$ 
\end{proposition}

\begin{proof}
Let $p\in \mathbb{OP}_l(\mathbb{X}).$ Then $0\leq p\leq e^l$ so that $\Vert p\Vert_l\leq \max \lbrace \Vert 0\Vert_l,\Vert e^l\Vert_l \rbrace=\Vert e^l\Vert_l=1.$ Next, by Proposition \ref{7}, we also have $0\leq p \leq \Vert p \Vert_l p.$ Thus $\Vert p\Vert_l\leq \max \lbrace \Vert 0\Vert_l,\Vert \Vert p\Vert_l p \Vert_l \rbrace=\Vert p\Vert_l^2$ and consequently $1\leq \Vert p\Vert_l$ for $p\neq 0.$ Finally, we conclude that $\Vert p \Vert_l = 1.$
\end{proof}

We generalize the example \ref{8} for order projections in absolute matrix order unit spaces by this immediate consequence.  

\begin{corollary}
Let $\mathbb{X}$ be an absolute matrix order unit space and $p\in \mathbb{OP}(\mathbb{X})\setminus \lbrace 0\rbrace,$ then $\mathbb{X}_p$ forms an absolute matrix order unit ideal of $\mathbb{X}.$ 
\end{corollary}

The notion of absolute and completely absolute value presrving maps have been introduced and studied in \cite{K19}. The following result tells that kernels of completely absolute value preserving maps are absolute matrix order ideals.

\begin{theorem}
Let $\mathbb{X}$ and $\mathbb{Y}$ be absolutely matrix ordered spaces and $\varphi:\mathbb{X}\to \mathbb{Y}$ be a completely absolute value preserving map. Put $Ker(\varphi)=\lbrace x\in \mathbb{X}:\varphi(x)=0\rbrace.$ Then the following statements are true:
\begin{enumerate}
\item[(1)] $\mathbb{M}_{l,m}(Ker(\varphi))=Ker(\varphi_{l,m})$ for every pair $l,m\in \mathbb{N}.$
\item[(2)] $Ker(\varphi_l)$ is self-adjoint for every $l\in \mathbb{N}.$
\item[(3)] $Ker(\varphi)$ is an absolutely matrix order ideal of $\mathbb{X}.$ 
\item[(4)] $\varphi=0$ if and only if $Ker(\varphi)=\mathbb{X}$ if and only if $Ker(\varphi)^+=\mathbb{X}^+.$
\end{enumerate}
Moreover, if $X$ is absolute matrix order unit space, then
\begin{enumerate}
\item[(5)] $\varphi=0$ if and only if $\varphi(e)=0.$
\end{enumerate}
\end{theorem}

\begin{proof}
\begin{enumerate}
\item[(1)] Let $[x_{i,j}]\in \mathbb{M}_{l,m}(Ker(\varphi)).$ Then $\varphi(x_{i,j})=0$ for every pair of $i$ and $j$ and consequently $\varphi_{l,m}([x_{i,j}])=[\varphi(x_{i,j})]=0.$ Thus $[x_{i,j}]\in Ker(\varphi_{l,m})$ so that $\mathbb{M}_{l,m}(Ker(\varphi))\subset Ker(\varphi_{l,m}).$ Tracing back the proof, we also conclude that $Ker(\varphi_{l,m})\subset \mathbb{M}_{l,m}(Ker(\varphi)).$ Hence $\mathbb{M}_{l,m}(Ker(\varphi))=Ker(\varphi_{l,m}).$
\item[(2)] $\varphi$ is absolute value preserving, we have $\varphi_l(x^*)=\varphi_l(x)^*$ for every $x \in \mathbb{M}_l(\mathbb{X})$ and $l\in \mathbb{N}.$ Now, the result follows.
\item[(3)] Put $\overline{\varphi}_l(x)=\varphi_l(x)$ for all $x\in \mathbb{M}_l(\mathbb{X})_{sa}.$ Then $\overline{\varphi}_l:\mathbb{M}_l(\mathbb{X})_{sa}\to \mathbb{M}_l(\mathbb{Y})_{sa}$ defines an absolute value preserving map with $Ker(\overline{\varphi}_l)=Ker(\varphi_l)_{sa}.$ By \cite[Theorem 2.7]{K19}, $Ker(\overline{\varphi}_l)$ is an order ideal of $\mathbb{M}_l(\mathbb{X})_{sa}.$ Thus $Ker(\phi)$ forms a matrix order ideal of $\mathbb{X}.$ Next, let $x=[x_{i,j}]\in \mathbb{M}_{l,m}(Ker(\varphi)).$ Without loss of genearlity, assume that $l\geq m$ so that $[x,0]\in \mathbb{M}_{l-m}(Ker(\varphi)).$ Then $\varphi_l([x,0])=0$ and consequently by Proposition \ref{9}(4), we have $\varphi_m(\vert x\vert_{l,m}) \circledast 0_{l-m}=\varphi_l(\vert x\vert_{l,m} \circledast 0_{l-m})=\varphi_l(\vert [x,0]\vert_l)=\vert \varphi_l([x,0])\vert_l=0.$ Thus $\vert x\vert_{l,m} \in \mathbb{M}_m(Ker(\varphi))^+$ so that $\vert \cdot \vert_{l,m}:\mathbb{M}_{l,m}(Ker(\varphi))\to \mathbb{M}_m(\varphi)^+$ is well defined for every pair of $l,m\in \mathbb{N}.$ Similarily, the case $l<m$ can be taken care. Hence the result follows.
\item[(4)] Let $\varphi=0.$ Then $Ker(\varphi)=\mathbb{X}$ so that $Ker(\varphi)^+=\mathbb{X}^+.$ Conversely, assume that $Ker(\varphi)^+=\mathbb{X}^+.$ Let $x\in \mathbb{X}_{sa}$ and $x=x^+-x^-$ be the orthogonal decomposition of $x.$ Since $\varphi(x^+)=\varphi(x^-)=0,$ we get $\varphi(x)=\varphi(x^+)-\varphi(x^-)=0.$ Next, let $x\in \mathbb{X}.$ Put $x_1=\frac{x-x^*}{2}$ and $x_2=\frac{x-x^*}{2i}$ so that $x_1,x_2\in \mathbb{X}_{sa}$ with $x=x_1+ix_2.$ Then $\varphi(x_1)=\varphi(x_2)=0$ and consequently $\varphi(x)=\varphi(x_1)+i\varphi(x_2)=0.$ Hence $\varphi=0.$
\end{enumerate}

If $X$ is an absolute matrix order unit space, then

\begin{enumerate}
\item[(5)] $\varphi=0$ implies $\varphi(e)=0.$ Conversely, assume that $\varphi(e)=0.$ Then $e\in Ker(\varphi).$ Now, let $x\in \mathbb{X}^+$ so that $0\leq x\leq \Vert x\Vert e.$ Since $Ker(\varphi)$ is an order ideal, we get $x\in Ker(\varphi)~i.e.~\varphi(x)=0.$ Thus $\phi(x)=0$ for all $x\in \mathbb{X}^+$ so that $Ker(\varphi)^+=\mathbb{X}^+.$ Hence, by (4), the result follows. 
\end{enumerate}
\end{proof}

Finally, we prove that $\frac{\mathbb{X}}{Ker(\varphi)},$ the quotient of $\mathbb{X}$ by $Ker(\varphi),$ is always an absolutely matrix ordered space. However, under the assumption of absolute matrix order unit property, it is identified with $\varphi(\mathbb{X})$ as an absolute matrix order unit ideal of $\mathbb{Y}.$ 

\begin{theorem}\label{13}
Let $\mathbb{X}$ and $\mathbb{Y}$ be absolutely matrix ordered spaces and $\varphi:\mathbb{X}\to \mathbb{Y}$ be a completely absolute value preserving map. Put $\mathbb{X}_0=\frac{\mathbb{X}}{Ker(\varphi)}.$ Then 
\begin{enumerate}
\item[(1)] $\mathbb{M}_l(\mathbb{X}_0)=\frac{\mathbb{M}_l(\mathbb{X})}{Ker(\varphi_l)}.$
\end{enumerate}
We write: $\mathbb{M}_l(\mathbb{X}_0)^+=\lbrace Ker(\varphi_l)+x:x\in \mathbb{M}_l(\mathbb{X})^+\rbrace.$ 
\begin{enumerate}
\item[(2)] $Ker(\varphi_l)+x\in \mathbb{M}_l(\mathbb{X}_0)^+$ if and only if $\varphi_l(x)\geq 0.$
\item[(3)] $\vert \cdot\vert_{l,m}^0:\mathbb{M}_{l,m}(\mathbb{X}_0)\to \mathbb{M}_l(\mathbb{X}_0)^+$ given by $Ker(\varphi_{l,m})+x\longmapsto Ker(\varphi_l)+\vert x\vert_{l,m}$ are well defined maps.
\item[(4)] $\left(\mathbb{X}_0,\lbrace \mathbb{M}_l(X_0)^+\rbrace,\lbrace \vert \cdot\vert^0_{l,m}\rbrace \right)$ is an absolutely matrix ordered space identified with $\varphi(\mathbb{X}).$  
\end{enumerate}
Moreover, if $\mathbb{X}~\textrm{and}~\mathbb{Y}$ are absolute matrix order unit spaces with $\Vert \varphi(e)\Vert=1$ and $\varphi(e)\in \vert \mathcal{M}\vert_\infty(\mathbb{Y}),$ then we also have: 
\begin{enumerate}
\item[(5)]$\mathbb{X}_0$ is an absolute matrix order ideal of $\mathbb{Y}$ as identified with $\mathbb{Y}_{\varphi(e)}=\varphi(\mathbb{X}).$
\end{enumerate}
\end{theorem}

\begin{proof}
\begin{enumerate}
\item[(1)] Observe that 
\begin{eqnarray*}
\mathbb{M}_l(\mathbb{X}_0) \ni \begin{bmatrix} Ker(\varphi)+x_{i,j}\end{bmatrix} &=& \lbrace \begin{bmatrix} z_{i,j}+x_{i,j}\end{bmatrix}:z_{i,j}\in Ker(\varphi)\rbrace \\
&=& \lbrace \begin{bmatrix} z_{i,j}\end{bmatrix}:z_{i,j}\in Ker(\varphi)\rbrace + \begin{bmatrix} x_{i,j}\end{bmatrix}\\
&=& Ker(\varphi_l)+ \begin{bmatrix} x_{i,j}\end{bmatrix}\in \frac{\mathbb{M}_l(\mathbb{X})}{Ker(\varphi_l)}.
\end{eqnarray*}

Thus $\mathbb{M}_l(\mathbb{X}_0)=\frac{\mathbb{M}_l(\mathbb{X})}{Ker(\varphi_l)}$ so that $\mathbb{M}_l(\mathbb{X}_0)^+=\lbrace Ker(\varphi_l)+x:x\in \mathbb{M}_l(\mathbb{X})^+\rbrace.$

\item[(2)] Let $Ker(\varphi_l)+x\in \mathbb{M}_l(\mathbb{X}_0)^+.$ Then there exists $z\in \mathbb{M}_l(\mathbb{X})^+$ such that $Ker(\varphi_l)+x=Ker(\varphi_l)+z.$ Thus $x-z\in Ker(\varphi_l)$ and we have $\varphi_l(x)=\varphi_l(z)\geq 0.$ Conversely, let $\varphi_l(x)\geq 0.$ For  $x_1=\frac{x-x^*}{2}$ and $x_2=\frac{x-x^*}{2i}\in \mathbb{M}_l(\mathbb{X})_{sa}$ such that $x_1+ix_2=x,$ we have $\varphi(x_1)+i\varphi(x_2)=\varphi(x)\geq 0.$ In this case, we get $\varphi_l(x_2)=0$ so that $Ker(\varphi_l)+x=Ker(\varphi_l)+x_1.$ Therefore, without loss of generality, we assume that $x \in \mathbb{M}_l(\mathbb{X})_{sa}.$ Next, let $x=x^+-x^-$ be the orthogonal decomposition of $x.$ By \cite[Proposition 2.6]{K19}, we get that $\varphi_l(x)^+=\varphi_l(x^+)$ and $\varphi_l(x)^-=\varphi_l(x^-).$ Since $\varphi_l(x)\geq 0$ and orthogonal decomposition is always unique, we conclude that $\varphi_l(x^-)=0.$ Thus $Ker(\varphi_l)+x=Ker(\varphi_l)+x^+ \in \mathbb{M}_l(\mathbb{X}_0)^+.$ 

\item[(3)] Let $Ker(\varphi_{l,m})+x=Ker(\varphi_{l,m})+y.$ Then $x-y\in Ker(\varphi_{l,m})$ so that $\varphi_{l,m}(x)=\varphi_{l,m}(y).$ Without loss of generality, we assume that $l\geq m.$ Since $\varphi_l$ is absolute value preserving, by Proposition \ref{9}(4), we get
\begin{eqnarray*}
\varphi_m(\vert x\vert_{l,m})\circledast 0 &=& \varphi_l(\vert \begin{bmatrix}x & 0\end{bmatrix}\vert_l)\\
&=& \vert \varphi_l(\begin{bmatrix} x & 0\end{bmatrix} ) \vert_l \\
&=& \vert \begin{bmatrix} \varphi_{l,m}(x) & 0\end{bmatrix}  \vert_l\\
&=& \vert \begin{bmatrix} \varphi_{l,m}(y) & 0\end{bmatrix}  \vert_l\\
&=& \varphi_m(\vert y\vert_{l,m})\circledast 0.
\end{eqnarray*}

and consequently $\varphi_m(\vert x\vert_{l,m})=\varphi_m(\vert y\vert_{l,m}).$ Thus $\vert x\vert_{l,m}-\vert y\vert_{l,m}\in Ker(\varphi_m)$ so that $Ker(\varphi_m)+\vert x\vert_{l,m}=Ker(\varphi_m)+\vert y\vert_{l,m}.$ Hence the maps $\vert \cdot\vert^0_{l,m}$ are well-defined.

\item[(4)] Let $x\in \mathbb{M}_{k,l}(\mathbb{X})$ and $y\in \mathbb{M}_{m,n}(\mathbb{X}).$ Then by \ref{5}(2)(b), we get
\begin{eqnarray*}
\vert (Ker(\varphi_{l,m})+x)\circledast (Ker(\varphi_{l,m})+y)\vert^0_{k+m,l+n} &=& \vert Ker(\varphi_{k+m,l+n})+(x\circledast y)\vert_{k+m,l+n}^0 \\
&=& Ker(\varphi_{l+n})+(\vert x\circledast y\vert_{k+m,l+n})\\
&=& Ker(\varphi_{l+n})+(\vert x\vert_{k,l}\circledast \vert y\vert_{m,n})\\
&=& (Ker(\varphi_l)+\vert x\vert_{k,l})\circledast (Ker(\varphi_n)+\vert y\vert_{m,n})\\
&=& \vert Ker(\varphi_{k,l})+x\vert_{k,l}\circledast \vert Ker(\varphi_{m,n})+y\vert_{m,n}.
\end{eqnarray*}

Next, let $\varsigma_1 \in \mathbb{M}_{k,m}$ and $\varsigma_2 \in \mathbb{M}_{n,l}.$ Again by \ref{5}(2)(a), we have:

\begin{eqnarray*}
\vert \varsigma_1 (Ker(\varphi_{m,n})+y)\varsigma_2\vert_{k,l}^0 &=& \vert Ker(\varphi_{k,l})+\varsigma_1 y\varsigma_2\vert_{k,l}^0\\
&=& Ker(\varphi_l)+\vert \varsigma_1 y\varsigma_2\vert_{k,l}^0\\
&\leq& Ker(\varphi_l)+\Vert \varsigma_1\Vert \vert \vert y\vert_{m,n}\varsigma_2\vert_{n,l}^0\\
&=& \Vert \varsigma_1\Vert \vert \vert Ker(\varphi_{m,n})+ y\vert_{m,n}^0\varsigma_2\vert_{n,l}^0.
\end{eqnarray*}

Thus $\mathbb{X}_0$ is an absolutely matrix ordered space. Now, $\mathbb{X}_0$ is identified with $\varphi(\mathbb{X})$ is immediately followed by first isomorphism theorem for vector spaces.
\end{enumerate}

Assume that $\mathbb{X}~\textrm{and}~\mathbb{Y}$ are absolute matrix order unit spaces with $\Vert \varphi(e)\Vert=1$ and $\varphi(e)\in \vert \mathcal{M}\vert_\infty(\mathbb{Y}).$ Then

\begin{enumerate}
\item[(5)] For $x\in \mathbb{M}_l(\mathbb{X})_{sa},$ there exists $\epsilon >0$ such that $\epsilon e^l\pm x\in \mathbb{M}_l(\mathbb{X})^+.$ Then  $\epsilon (Ker(\varphi)+e)^l\pm (Ker(\varphi_l)+x)=\epsilon (Ker(\varphi_l)+e^l)+(Ker(\varphi_l)\pm x)=Ker(\varphi_l)+(\epsilon e^l\pm x)\in \mathbb{M}_l(\mathbb{X}_0)^+.$ Thus $Ker(\varphi)+e$ is the order unit for $\mathbb{X}_0.$

Let $Ker(\varphi_l)\pm x\geq 0.$ By (2), we have $\pm \varphi_l(x)= \varphi_l(\pm x)\geq 0.$ Since $\mathbb{M}_l(\mathbb{X})^+$ is proper, we get that $\varphi_l(x)=0~i.e.~Ker(\varphi_l)+ x= 0.$ Thus $\mathbb{M}_l(\mathbb{X}_0)^+$ is proper for every $l\in \mathbb{N}.$

Let $x\geq 0$ with $\epsilon (Ker(\varphi_l)+x)+(Ker(\varphi_l)+z)\geq 0$ for all $\epsilon >0.$ Again, by (2), we have $\epsilon \varphi_l(x)+\varphi_l(z)=\varphi_l(\epsilon x+z)\geq 0.$  Since $\varphi_l(x)\geq 0~\textrm{and}~\mathbb{M}_l(\mathbb{X})^+$ is Archimedean, we get that $\varphi_l(z)\geq 0.$ Then $Ker(\varphi_l)+z\geq 0.$ Therefore $\mathbb{M}_l(\mathbb{X}_0)^+$ is Archimedean for every $l\in \mathbb{N}.$

As $\mathbb{M}_l(\mathbb{X})_0^+$ is proper and Archimedean for all $l\in \mathbb{N}.$ Therefore order unit $Ker(\varphi)+e$ determines a matrix norm $\lbrace \Vert \cdot \Vert_l^0\rbrace$ on $\mathbb{X}_0$ defined in the following way:

\begin{eqnarray*}
\Vert Ker(\varphi_l)+x\Vert_l^0 &=& \inf \left \lbrace \epsilon >0: Ker(\varphi_{2l})+ \left (\begin{bmatrix}\epsilon e^l & \pm x\\ \pm x^* & \epsilon e^l\end{bmatrix}\right)\geq 0\right\rbrace \\
&=& \inf \left \lbrace \epsilon >0: \begin{bmatrix}\epsilon \varphi(e)^l & \pm \varphi_l(x)\\ \pm \varphi_l(x)^* & \epsilon \varphi(e)^l\end{bmatrix} \geq 0\right \rbrace
\end{eqnarray*}

for every $l \in \mathbb{N}.$

Finally, assume that $\varphi(e)\in \vert \mathcal{M} \vert_\infty(Y)$ such that $\Vert \varphi(e)\Vert=1.$ In this case, we show that $\mathbb{X}_0\cong \mathbb{Y}_{\varphi(x)}=\varphi(\mathbb{X}).$ For that let $Ker(\varphi_l)+x\geq 0.$ Then by Theorem \ref{4}, we get that

\begin{eqnarray*}
\Vert Ker(\varphi_l)+x\Vert_l^0 &=& \inf \lbrace \epsilon >0: \epsilon (Ker(\varphi)+e)\pm (Ker(\varphi)+x) \geq 0\rbrace \\
&=& \inf \lbrace \epsilon >0: Ker(\varphi)+\epsilon e\pm x \geq 0\rbrace \\
&=& \inf \lbrace \epsilon >0: \epsilon \varphi(e)\pm \varphi(x) \geq 0\rbrace\\
&=& \Vert \varphi(x)\Vert_l^{\varphi(e)}\\
&=& \Vert \varphi(x)\Vert_l.
\end{eqnarray*}
\end{enumerate}
Hence the result follows.
\end{proof}

\section{Relation between $\mathcal{K}_0(\mathbb{X}_x)$ and $\mathcal{K}_0(\mathbb{X})$}

In this section, we derive a relation between $\mathcal{K}_0(\mathbb{X}_x)$ and $\mathcal{K}_0(\mathbb{X}).$ For that, we start with the following characterization of order projections in $\mathbb{X}_x$ in terms of order projections in $\mathbb{X}.$

\begin{lemma}
In an absolute matrix order unit space $\mathbb{X},$ let $x\in \mathbb{X}^+$ such that $\Vert x\Vert=1.$ Then $\mathbb{OP}_l(\mathbb{X}_x)  =  \lbrace p \in \mathbb{OP}_l(\mathbb{X}): p \leq x^l \rbrace =  \mathbb{OP}_l(\mathbb{X}) \cap \mathbb{M}_l(\mathbb{X}_x)$ for every $l \in \mathbb{N}.$
\end{lemma}

\begin{proof}
Let $p\in \mathbb{OP}_l(\mathbb{X}) \cap \mathbb{M}_l(\mathbb{X}_x).$ By proposition \ref{1}, we get that $0\leq p\leq x^l.$ For $\Vert x\Vert =1,$ we also get that $0\leq x^l \leq e^l.$ Then $0\leq x^l-p\leq e^l-p.$ Since $p\in \mathbb{OP}_l(\mathbb{X}),$ we have $e^l\perp e^l-p.$ By \cite[Definition 3.4(4)]{K18}, we conclude $p\perp x^l-p.$ Thus $p\in \mathbb{OP}_l(\mathbb{X}_x).$ Hence the reslut follows.
\end{proof}

\begin{corollary}
In an absolute matrix order unit space $\mathbb{X},$ we have: $\mathbb{OP}_l(\mathbb{X}_p)  =  \lbrace q \in \mathbb{OP}_l(\mathbb{X}): q \leq p^l \rbrace =  \mathbb{OP}_l(\mathbb{X}) \cap \mathbb{M}_l(\mathbb{X}_p)$ for every $p\in \mathbb{OP}(\mathbb{X})\setminus \lbrace 0\rbrace$ and $l\in \mathbb{N}.$ 
\end{corollary}

The partial isometric equivalence (denoted by $\sim$) on order projections is defined in \cite{PI19}. Now, we characterize $\sim$ in $\mathbb{X}_x$ in terms of $\sim$ in $\mathbb{X}.$ 

\begin{proposition}\label{2}
In an absolute matrix order unit space $\mathbb{X},$ let $p,q \in \mathbb{OP}_\infty(\mathbb{X}_x)$ for some $x \in \mathbb{X}^+$ such that $\Vert x\Vert=1.$ Then $p \sim q$ in $\mathbb{X}_x$ if and only if $p \sim q$ in $\mathbb{X}.$
\end{proposition}

\begin{proof}
Suppose that $p \sim q$ in $\mathbb{X}$ for some $p \in \mathbb{OP}_l(\mathbb{X}_x)$ and $q \in \mathbb{OP}_m(\mathbb{X}_x).$ There exists $y \in \mathbb{M}_{l,m}(\mathbb{X})$ satisfying $p = \vert y^* \vert_{m,l}$ and $q = \vert y \vert_{l,m}.$ For $\Vert p\Vert\leq 1$ and $\Vert q\Vert \leq 1,$ we have $p\leq x^l$ and $q\leq x^m.$ Then
\begin{eqnarray*}
\begin{bmatrix} x^l  & 0 \\ 0 & x^m \end{bmatrix} &\geq & \begin{bmatrix} p & 0 \\ 0 & q \end{bmatrix} \\
&=& \begin{bmatrix}\vert y^* \vert_{m,l} & 0 \\ 0 &  \vert y \vert_{l,m} \end{bmatrix} \\
&=& \left \vert \begin{bmatrix} 0 & y \\ y^* & 0\end{bmatrix} \right \vert_{l+m} \\
&\ge& \pm \begin{bmatrix} 0 & y \\ y^* & 0 \end{bmatrix}
\end{eqnarray*}

so that $\begin{bmatrix} x^l & y \\ y^* & x^m \end{bmatrix} \in \mathbb{M}_{l+m}(\mathbb{X})^+.$ Thus $y \in \mathbb{M}_{l,m}(\mathbb{X}_x)$ and consequently $p \sim q$ in $\mathbb{X}_x.$

Next, $p \sim q$ in $\mathbb{X}_x$ implies $p \sim q$ in $\mathbb{X}$ follows trivially. 
\end{proof}

The corresponding characterization for partial isometrises of $\mathbb{X}_x$ is also given in terms of partial isometrises of $\mathbb{X}$ by the following result.

\begin{corollary}\label{3}
In an absolute matrix order unit space $\mathbb{X},$ let $x \in \mathbb{X}^+$ such that $\Vert x\Vert=1.$ If $y \in \mathbb{PI}_{l,m}(\mathbb{X})$ such that $\vert y\vert_{l,m}\in \mathbb{OP}_m(\mathbb{X}_x)$ and $\vert y^*\vert_{m,l}\in \mathbb{OP}_l(\mathbb{X}_x),$ then $x\in \mathbb{PI}_{l,m}(\mathbb{X}_x).$ In this case, $\mathbb{PI}_{l,m}(\mathbb{X}_x)=\mathbb{M}_{l,m}(\mathbb{X}_x)\cap \mathbb{PI}_{l,m}(\mathbb{X}).$ In particular, $\mathbb{PI}_{l,m}(\mathbb{X}_p)=\mathbb{M}_{l,m}(\mathbb{X}_p)\cap \mathbb{PI}_{l,m}(\mathbb{X})$ for every $p\in \mathbb{OP}(\mathbb{X})\setminus \lbrace 0\rbrace.$ 
\end{corollary}

\begin{proof}
Let $y \in \mathbb{PI}_{l,m}(\mathbb{X})$ such that $\vert y\vert_{l,m}\in \mathbb{OP}_m(\mathbb{X}_x)$ and $\vert y^*\vert_{m,l}\in \mathbb{OP}_l(\mathbb{X}_x).$ Put $p=\vert y^*\vert_{m,l}$ and $q=\vert y\vert_{l,m}.$ Then $p\sim q$ in $\mathbb{X}_x.$ Now the result follows immediately from the Proposition \ref{2}.
\end{proof}

The condition $(T)$ is also defined in \cite{PI19}. Under $(T),$ the relation $\sim$ becomes an equivalence relation on order projections. In the next result, it is shown that $(T)$ is tranfered on $\mathbb{X}_x$ as a heredity. 

\begin{proposition}
Let $\mathbb{X}$ an absolute matrix order unit space and let $x \in \mathbb{X}^+$ such that $\Vert x\Vert=1.$ Then (T) holds in $\mathbb{X}$ implies (T) holds in $\mathbb{X}_x.$ In particular, (T) holds in $\mathbb{X}$ implies (T) holds in $\mathbb{X}_p$ for every $p\in \mathbb{OP}(\mathbb{X})\setminus \lbrace 0\rbrace.$
\end{proposition}

\begin{proof}
Assume that (T) holds in $\mathbb{X}.$ Let $y \in \mathbb{PI}_{m,n}(\mathbb{X}_x)$ and $z \in \mathbb{PI}_{l,n}(\mathbb{X}_x)$ satisfying $\vert y \vert_{m,n} = \vert z \vert_{l,n}.$ Then there exists $w \in \mathbb{PI}_{m,l}(\mathbb{X})$ such that $\vert w^* \vert_{l,m} = \vert y^* \vert_{n,m}$ and $\vert w \vert_{m,l} = \vert z^* \vert_{n,l}.$ Note that $\vert w \vert_{m,l}\in\mathbb{OP}_l(\mathbb{X}_x)$ and $\vert w^* \vert_{l,m}\in \mathbb{OP}_m(\mathbb{X}_x).$ By Corollary \ref{3}, we get $w \in \mathbb{PI}_{m,l}(\mathbb{X}_x).$ Thus (T) also holds in $\mathbb{X}_x.$
\end{proof}

Finally, we derive a relation between $\mathcal{K}_0(\mathbb{X}_x)$ and $\mathcal{K}_0(\mathbb{X})$ by the following result.

\begin{theorem}\label{14}
Let $(T)$ holds in an absolute matrix order unit space $\mathbb{X}$ and $x\in \mathbb{X}^+$ such that $\Vert x\Vert=1.$ Then $[(p,q)] \longmapsto [(p,q)]$ defines a group homomorphism  $\mathcal{K}_0(\mathbb{X}_x)$ to $\mathcal{K}_0(\mathbb{X}).$ Moreover, if $p\approx q$ in $\mathbb{OP}_\infty(\mathbb{X})$ implies $p\approx q$ in $\mathbb{OP}_\infty(\mathbb{X}_x)$ for every pair of $p,q\in \mathbb{OP}_\infty(\mathbb{X}_x),$ then this group homomorphism turns out to be an injection.
\end{theorem}

\begin{proof}
Consider the inclusion map $i: \mathbb{X}_x \hookrightarrow \mathbb{X}.$ By \cite[Theorem 5.2]{KO21}, there exists a unique group homomorphism $\mathcal{K}_0(i): \mathcal{K}_0(\mathbb{X}_x) \to \mathcal{K}_0(\mathbb{X})$ satisfying the following commutative diagram: 

$$\begin{tikzcd}
\mathbb{OP}_\infty(\mathbb{X}_x)\arrow[r, hook, "i"]\arrow[swap]{d}{\chi_{\mathbb{X}_x}}
& \mathbb{OP}_\infty(\mathbb{X})\arrow{d}{\chi_\mathbb{X}} \\
\mathcal{K}_0(\mathbb{X}_x)\arrow[swap]{r}{\mathcal{K}_0(i)}
& \mathcal{K}_0(\mathbb{X})
\end{tikzcd}$$ so that $$\mathcal{K}_0(i)(\chi_{\mathbb{X}_x}(p))=\chi_\mathbb{X}(p)~\textrm{for every}~p \in \mathbb{OP}_\infty(\mathbb{X}_x).$$ 

Now, for $[(p,q)] \in \mathcal{K}_0(\mathbb{X}_x),$ we have: 

\begin{eqnarray*}
\mathcal{K}_0(i)([(p,q)]) &=& \mathcal{K}_0(i)([(p,0)])-[(0,q)])\\
&=& \mathcal{K}_0(i)(\chi_{\mathbb{X}_x}(p)-\chi_{\mathbb{X}_x}(q)) \\
&=& \mathcal{K}_0(i)(\chi_{\mathbb{X}_x}(p))-\mathcal{K}_0(i)(\chi_{\mathbb{X}_x}(q)) \\
&=& \chi_\mathbb{X}(p)-\chi_\mathbb{X}(q) \\
&=& [(p,q)].
\end{eqnarray*}

Thus $[(p,q)] \longmapsto [(p,q)]$ defines a group homomorphism from $\mathcal{K}_0(\mathbb{X}_x)$ to $\mathcal{K}_0(\mathbb{X})$. 

Next, let $p,q \in \mathbb{OP}_\infty(\mathbb{X}_x)$ such that $[(p,q)]=0$ in $\mathbb{X}.$ In this case, we have $p\oplus 0\approx 0\oplus q$ in $\mathbb{X}.$ By assumption, we conclude that $p\oplus 0\approx 0\oplus q$ in $\mathbb{X}_x.$ Thus $[(p,q)]=0$ in $\mathbb{X}_x.$ Hence the group homomorphism is injective.
\end{proof}

\begin{corollary}
Let $(T)$ holds in an absolute matrix order unit space $\mathbb{X}$ and $p\in \mathbb{OP}(\mathbb{X})\setminus \lbrace 0\rbrace.$ Then $[(q,r)] \longmapsto [(q,r)]$ defines a group homomorphism  $\mathcal{K}_0(\mathbb{X}_p)$ to $\mathcal{K}_0(\mathbb{X}).$ Moreover, if $q\approx r$ in $\mathbb{X}$ implies $q\approx r$ in $\mathbb{X}_p$ for every pair $q,r\in \mathbb{OP}_\infty(\mathbb{X}_p),$ then this group homomorphism turns out to be an injection.
\end{corollary}

\end{document}